\documentclass[11pt,leqno]{amsart}
\usepackage{amssymb,amsmath,amsthm,amsfonts}
\usepackage{latexsym}
\usepackage{color}
\allowdisplaybreaks

\newcommand{\Hmm}[1]{\leavevmode{\marginpar{\tiny%
$\hbox to 0mm{\hspace*{-0.5mm}$\leftarrow$\hss}%
\vcenter{\vrule depth 0.1mm height 0.1mm width \the\marginparwidth}%
\hbox to 0mm{\hss$\rightarrow$\hspace*{-0.5mm}}$\\\relax\raggedright #1}}}
\newcommand{\nc}{\newcommand}
\nc{\les}{\lesssim}
\nc{\ges}{\gtrsim}
\nc{\nit}{\noindent}
\nc{\nn}{\nonumber}
\nc{\D}{\partial}
\nc{\diff}[2]{\frac{d #1}{d #2}}
\nc{\diffn}[3]{\frac{d^{#3} #1}{d {#2}^{#3}}}
\nc{\pdiff}[2]{\frac{\partial #1}{\partial #2}}
\nc{\pdiffn}[3]{\frac{\partial^{#3} #1}{\partial{#2}^{#3}}}
\nc{\abs}[1] {\lvert #1 \rvert}
\nc{\cAc}{{\cal A}_c}
\nc{\cE}{{\cal E}}
\nc{\cF}{{\mathcal F}}
\nc{\cP}{{\cal P}}
\nc{\cV}{{\cal V}}
\nc{\cQ}{{\cal Q}}
\nc{\cGin}{{\cal G}_{\rm in}}
\nc{\cGout}{{\cal G}_{\rm out}}
\nc{\cO}{{\cal O}}
\nc{\Lav}{{\cal L}_{\rm av}}
\nc{\cL}{{\cal L}}
\nc{\cB}{{\cal B}}
\nc{\cZ}{{\cal Z}}
\nc{\cR}{{\cal R}}
\nc{\cT}{{\cal T}}
\nc{\cY}{{\cal Y}}
\nc{\cX}{{\cal X}}
\nc{\cXT}{{{\cal X}(T)}}
\nc{\cBT}{{{\cal B}(T)}}
\nc{\vD}{{\vec \mathcal{D}}}
\nc{\efield}{\mathcal{E}}
\nc{\vE}{{\vec \efield}}
\nc{\vB}{{\vec \mathcal{B}}}
\nc{\vH}{{\vec \mathcal{H}}}
\nc{\ty}{{\tilde y}}
\nc{\tu}{{\tilde u}}
\nc{\tV}{{\tilde V}}
\nc{\Pc}{{\bf P_c}}
\nc{\bx}{{\bf x}}
\nc{\bX}{{\bf X}}
\nc{\bXYZ}{{\bf XYZ}}
\nc{\bY}{{\bf Y}}
\nc{\bF}{{\bf F}}
\nc{\bS}{{\bf S}}
\nc{\dV}{{\delta V}}
\nc{\dE}{{\delta E}}
\nc{\TT}{{\Theta}}
\nc{\dPsi}{{\delta\Psi}}
\nc{\order}{{\cal O}}
\nc{\Rout}{R_{\rm out}}
\nc{\eplus}{e_+}
\nc{\eminus}{e_-}
\nc{\epm}{e_\pm}
\nc{\eps}{\varepsilon}
\nc{\vnabla}{{\vec\nabla}}
\nc{\G}{\Gamma}
\nc{\w}{\omega}
\nc{\mh}{h}
\nc{\mg}{g}
\nc{\vphi}{\varphi}
\nc{\tlambda}{\tilde\lambda}
\nc{\be}{\begin{equation}}
\nc{\ee}{\end{equation}}
\nc{\ba}{\begin{eqnarray}}
\nc{\ea}{\end{eqnarray}}

\nc{\g}{\gamma}
\nc{\ol}{\overline}

\newtheorem{theorem}{Theorem}[section]
\newtheorem{lemma}[theorem]{Lemma}
\newtheorem{prop}[theorem]{Proposition}

\nc{\T}{\mathbb T}
\nc{\Z}{\mathbb Z}
\nc{\N}{\mathbb N}
\nc{\pt}{\partial_t}
\nc{\la}{\langle}
\nc{\ra}{\rangle}
\nc{\infint}{\int_{-\infty}^{\infty}}
\nc{\halfwidth}{6.5cm}
\nc{\figwidth}{10cm}

\nc{\nlayers}{L} \nc{\nsectors}{M}
\nc{\indicator}{\mathbf{1}}
\nc{\Rhole}{R_{\rm hole}}
\nc{\Rring}{R_{\rm ring}}
\nc{\neff}{n_{\rm eff}}
\nc{\Frem}{F_{\rm rem}}
\nc{\DD}{\Delta}
\nc{\cD}{\mathcal D}
\nc{\lnorm}{\left\|}
\nc{\rnorm}{\right\|}
\nc{\rnormp}{\right\|_{\ell^{p,\eps}}}
\nc{\rar}{\rightarrow}
\nc{\sgn}{{\rm sign}}
\nc{\non}{\nonumber}
\nc{\wh}{\widehat}
\allowdisplaybreaks
\date{\today}

\begin{document}

\title[Schr\"odinger equation on irrational tori]{Local Well-posedness for 2-D Schr\"odinger Equation on Irrational Tori and Bounds on Sobolev Norms}

\author{Seckin Demirbas}

\address{Department of Mathematics \\
University of Illinois \\
Urbana, IL 61801, U.S.A.}

\email{demirba2@illinois.edu }

\begin{abstract}

 In this paper we consider the cubic Schr\"odinger equation in two space dimensions on irrational tori. Our main result is an improvement of the Strichartz estimates on irrational tori. Using this estimate we obtain a local well-posedness result in $H^{s}$ for $s>\frac{131}{416}$. We also obtain improved growth bounds for higher order Sobolev norms.
\end{abstract}

\maketitle

\section{Introduction}

The equation we consider in this paper is the cubic, Schr\"{o}dinger equation on irrational tori, namely,

\begin{equation}\label{eqn:sch}
\left\{
\begin{array}{l}
iu_{t}+\Delta u =\alpha|u|^2u, \,\,\,\,  x \in {\mathbb T_\theta^2}, \,\,\,\,  t\in [-T,T],\\
u(x,0)=u_0(x)\in H^{s}(\mathbb T_\theta^2), \\
\end{array}
\right.
\end{equation}
where $\alpha=\pm 1$. $T$ is the time of existence of the solutions and $\mathbb T_\theta^2$ is the irrational tori, $\mathbb R^2/{\theta_1 \mathbb Z \times \theta_2 \mathbb Z}$ for $\theta_1, \theta_2 > 0$ and $\theta_1/\theta_2$ irrational. The equation is called focusing for $\alpha=-1$ and defocusing for $\alpha=1$.

 The equation \eqref{eqn:sch} posed $\mathbb{T}^2$, has been studied widely for its importance in the theory of differential equations. The equation is energy subcritical, see \cite{[tao]}, and for defocusing equation, for any initial data in $H^1$ there is global well-posedness and global bounds on the Sobolev norm of the solution, see \cite{[bou-gafa]}. In addition there have been many results on the well-posedness of \eqref{eqn:sch} for both focusing and defocusing case for rough initial data (in $H^s$ for $s<1$) on two dimensional torus, see \cite{[bou]},\cite{[bgt]},\cite{[cat-wang]}, \cite{[nikos]},\cite{[guo]} and also on more general compact manifolds, see \cite{[bgt-comp]},\cite{[zho]}.
 
One of the main tools in proving local well-posedness is the Strichartz estimates, i.e. estimates of the form 
\begin{equation}\label{str}
\|e^{it\Delta}f\|_{L_t^4L_x^4(\mathbb T\times\mathbb T_\theta^2)}\les \|f\|_{H^{\frac{s_0}{2}}(\mathbb T_\theta^2)},
\end{equation}
for some $s_0\geq 0$ and $f\in H^{\frac{s_0}{2}}(\mathbb T_\theta^2)$. Our main focus in this paper will be on the improvement of this estimate on irrational tori.  
 As one can see for $\theta_1=\theta_2=1$ we get the usual (flat) torus. Although the domain resembles the flat torus, the tools used to prove \eqref{str} are fundamentally different. The reason behind this difference is that the symbol of the Laplacian on flat torus at any $(m,n)$-level is $m^2+n^2$ whereas the symbol of it on a irrational torus is of the form $(\theta_1m)^2+(\theta_2n)^2$. Thus the method of counting lattice points on a circle to get \eqref{str} cannot be applied here. In 3-d, Bourgain \cite{[bou]}, used bounds on the $l^p$-norms on the number of lattice points on the ellipsoid and Jarnick's estimate \cite{[jar]} to get \eqref{str} with $s_0=\frac{1}{3}$. A slight modification of his method in 2-d gives us a $\frac{1}{4}$-derivative loss in \eqref{str}. But this result was already proven for not only on irrational tori but also on any two dimensional compact manifold, see \cite{[bgt-comp]}. This remedy was overcome in Catoire and Wang's paper \cite{[cat-wang]} using Jarnick's estimate, see \cite{[jar]}, in two dimensions without passing to the $l^p$-norms of the number of lattice points on ellipsoids. They obtained \eqref{str} with $\frac{s_0}{2}=\frac{1}{6}$. The first part of this paper will be consisted of our main result, improving \eqref{str} to $\frac{s_0}{2}=\frac{131}{832}$ using a counting argument of Huxley, \cite{[hux]}. In the second part of the paper, using the theory of Bourgain spaces, we prove local well-posedness for initial data in $H^{s}$, $s>s_0$ and also polynomial bounds on the growth of the Sobolev norms of the solution for the defocusing case. On 2-d flat tori, we should note that the local well-posedness theory gives the exponential bound $\|u(t)\|_{H^s} \les C^{|t|}$, see \cite{[bou-growth]}. Also in \cite{[bou-growth]}, Bourgain improved this exponential bound with the polynomial bound $\|u(t)\|_{H^s} \les C\langle t \rangle^{2(s-1)+}$ using the following polynomial estimate:
\begin{lemma}\label{poly-est}
If there exists a constant $r \in (0,1)$ and $\delta > 0$ such that for any time $t_0$ we have, $$\|u(t_0 + \delta)\|
_{H^s}^2\leq \|u(t_0)\|_{H^s}^2+C\|u(t_0)\|_{H^s}^{2-r},$$ then we get $$\|u(t)\|_{H^s} \les (1 + |t|)^{\frac{1}{r}}.$$
\end{lemma}
It suffices to prove this result for $t$ being an integer multiple of $\delta$ and the rest follows from induction. This result was later improved by Staffilani, see \cite{[staf]} to $\|u(t)\|_{H^s} \les C\langle t \rangle^{(s-1)+}$. On 2-d irrational tori, using Zhong's arguments \cite{[zho]} and the lemma above, Catoire and Wang proved the norm bound $\|u(t)\|_{H^s} \les C\langle t \rangle^{\frac{(s-1)}{1-\frac{2}{3}}+}$, see \cite{[cat-wang]}. In this paper we are going to improve the polynomial bound on 2-d irrational tori to the exponent $\frac{(s-1)}{(1-131/416)}+$.
\section{Notations}

  Throughout the paper, $L^2_tL^2_x$ will denote the space $L^2_tL^2_x(\mathbb T\times \mathbb T_\theta^2)$ and $L^2_tL^2_x([0,T])$ for $L^2_tL^2_x([0,T]\times \mathbb T_\theta^2)$, which will be used for Sobolev spaces too. 
  
  We will use $(.)^+$ to denote $(.)^\epsilon$ for all $\epsilon>0$ with implicit constants depending on $\epsilon$ and will use the usual Japanese bracket notation, $\langle x\rangle=(1+x^2)^{1/2}$.
  
  The linear propogator of the Schr\"odinger equation will be denoted as $e^{it\Delta}$, where it is defined on the Fourier side as $\widehat{(e^{it\Delta}f)}(m_1,m_2)=e^{-itQ(m_1,m_2)}\hat{f}(m_1,m_2)$, where $Q(m_1,m_2)=(\theta_1m_1)^2+(\theta_2m_2)^2$.
  
  The Bourgain spaces $X^{s,b}$ will be defined as the closure of compactly supported smooth functions under the norm $$\|u\|_{X^{s,b}}\dot{=}\|e^{-it\Delta}u\|_{H^b_t(\mathbb{R})H^s_x(\mathbb{T}_{\theta}^2)}=\|\langle \tau-Q(m,n) \rangle^{b} \langle |m|+|n| \rangle^s\widehat{u}(m,n,\tau)\|_{L_{\tau}^2 l^2_{(m,n)}},$$ and the restricted norm will be	 given as $$\|u\|_{X_T^{s,b}}\dot{=}\inf(\|v\|_{X^{s,b}},\ for\ v=u\ on\ [0,T]).$$ 
    
  For any operator $D$ and positive number $N$, $\mbox{\large$\chi$}$ being the characteristic function, $\mbox{\large$\chi$}^D_N u$ is defined to be $\mbox{\large$\chi$}_{D\in [N,2N]} u$.
  
  Also, in this paper we will use $s_0=\frac{131}{416}$.
  
\section{Preliminaries}

 When we say the equation \eqref{eqn:sch} is locally well-posed in $H^s$, we mean that there exist a time $T_{LWP}=T_{LWP}(\|u_0\|_{H^s})$ such that the solution exists and is unique in $X^{s,b}_{T_{LWP}}\subset C([0,T_{LWP}),H^s)$ and depends continuously on the initial data. We say that the the equation is globally well-posed when $T_{LWP}$ can be taken arbitrarily large.

 By Duhamel Principle, we know that the smooth solutions of \eqref{eqn:sch} satisfy the integral equation
 $$u(t,x)=e^{it\Delta}u_0(x)-i\alpha\int_0^t e^{i(t-\tau)\Delta}|u|^2u(\tau,x)d\tau.$$
 Thus, our main concern is to find the fixed point to the integral operator
  \begin{equation}
S(u)(t,x)=e^{it\Delta}u_0(x)-i\alpha\int_0^t e^{i(t-\tau)\Delta}|u|^2u(\tau,x)d\tau.\non
\end{equation}
To do that, we will use Banach Fixed Point Theorem on the set of functions
$$B_T=\{ u\in X^{s,b}: u(0,x)=u_0(x)\quad and\quad \|u\|_{X_T^{s,b}}\leq 2\|u_0\|_{H^s} \},$$
and get a contraction for sufficiently small $T$. 

We also note that the equation has mass and energy conservations, namely,
   $$M(u)(t)=\int_{\mathbb{T}_{\theta}^2}|u(t,x)|^2=M(u)(0),$$
and, $$E(u)(t)=\int_{\mathbb{T}_{\theta}^2}|\nabla u(t,x)|^2+\frac{\alpha}{2}\int_{\mathbb{T}_{\theta}^2}|u(t,x)|^4=E(u)(0).$$

  Thus, for the defocusing equation, i.e. $\alpha=1$, we have global bounds on the $H^1$-norm of the solution. This also says, for defocusing equation we have $H^1$ global well-posedness.

\section{Main Results}

 \begin{theorem}\label{local}

 The 2-d cubic Schr\"odinger equation \eqref{eqn:sch}
 is locally well-posed for initial data $u_0 \in H_x^s$ for $s>s_0$.

 \end{theorem}
 
 And we will also prove,
 
 \begin{theorem}\label{growth}

For $s\geq1$, let $u(t,x)$ be the solution to the defocusing cubic Schr\"odinger equation \eqref{eqn:sch}. Then for any time $t$, we have, $$\|u(t,x)\|_{H_x^s}\leq C \langle t\rangle^{\frac{(s-1)+}{(1-s_0)}}\|u_0\|_{H_x^s}.$$

 \end{theorem}

\subsection{Proof of The Results}

\subsubsection{Proof of Strichartz estimates}

  To be able to prove \eqref{str}, we will use a counting argument by Huxley, \cite{[hux]}:
  
  \begin{theorem}\label{hux}
 
  For $a,b,c \in \mathbb R$, let $Q = Q(m, n) = am^2+bmn+cn^2$ be a positive definite quadratic form, where $a > 0$, $D := 4ac - b^2 > 0$. For $x$ large, we have $$\#\{(m, n) \in \Z^2
: Q(m, n) \leq x\} = \frac{2\pi}{\sqrt{D}}x+R(x), $$ where $R(x)\leq x^{(\frac{131}{416})+}$.
  
  \end{theorem}
 
 \begin{theorem}
 Let $f\in L_x^2$ such that $supp(\widehat{f})\in B(0,N)$, then $$\|e^{it\Delta}f\|_{L^4_tL^4_x}\les N^{(\frac{s_0}{2})+}\|f\|_{L_x^2}.$$
 \end{theorem}
 
   \begin{proof}
   \begin{eqnarray}
   \|e^{it\Delta}f\|_{L^4_tL^4_x}^2 &=& \|(e^{it\Delta}f)^2\|_{L^2_tL^2_x}\non\\
   &=& \Big{\|} \big{[} \sum_{m \in \Z^2}|\sum_{n \in \mathbb Z^2} \widehat{f}(n)\widehat{f}(m-n)e^{-it(Q(n)+Q(m-n))}|^2 \big{]}^{1/2}\Big{\|}_{L_t^2}\non\\
   &=&\Big{[}\sum_{m\in \Z^2}\big{\|}\sum_{n\in \Z^2}\widehat{f}(n)\widehat{f}(m-n)e^{-it(Q(n)+Q(m-n))}\big{\|}_{L_t^2}^2\Big{]}^{1/2}\non\\
   &\les&\Big{[}\sum_{m \in \Z^2} \big{[}\sum_{k \in \Z}(\sum_{|Q(n)+Q(m-n)-k|\leq 1/2}|\widehat{f}(n)\widehat{f}(m-n)|)^2\big{]}\Big{]}^{1/2},\non
   \end{eqnarray}

  where, to pass to the last inequality we used:
  
  \begin{lemma}
  $\|\sum_n e^{ita_n}b_n\|_{L^2([0,1])}^2\les \sum_j (\sum_{|a_n-j|\leq 1/2} |b_n|)^2$.
  \end{lemma}
  \begin{proof}
 For any finite sum over $n$, write $\|\sum_n e^{ita_n}b_n\|_{L^2([0,1])}^2=\|\sum_j \sum_{|a_n-j|\leq1/2} e^{ita_n}b_n\|_{L^2([0,1])}^2$. Hence, for a bump function $\phi$ s.t. $\phi(t)=1$ in $[0,1]$ we have
  \begin{eqnarray}
  \|\sum_n b_ne^{ita_n}\|_{L^2[0,1]}^2 &\leq&\|\sum_n b_ne^{ita_n}\phi(t)\|_{L^2(\mathbb{R})}^2\non\\
  &=& \|\sum_n b_n\widehat{\phi}(\xi-a_n)\|_{L^2(\mathbb{R})}^2\non\\
  &\les& \|\sum_n |b_n|\frac{1}{\langle\xi-a_n\rangle^{\alpha}}\|_{L^2(\mathbb{R})}^2\non\\
  &=& \|\sum_j \sum_{|a_n-j|\leq1/2} |b_n|\frac{1}{\langle\xi-a_n\rangle^{\alpha}}\|_{L^2(\mathbb{R})}^2\non\\
  &\les& \|\sum_j \frac{1}{\langle \xi-j\rangle^{\alpha}}\sum_{|a_n-j|\leq1/2} |b_n|\|_{L^2(\mathbb{R})}^2\non\\
  &\leq& \sum_j (\sum_{|a_n-j|\leq 1/2} |b_n|)^2.\non
    \end{eqnarray}

    Here, to pass to the second line we used Plancherel's equality. In the third line we used that the Fourier transform of $\phi$ is a Schwartz function and decays faster than any polynomial, and thus we can choose an $\alpha>1$. Also to pass to the last line we used Young's inequality. 
  \end{proof}

  Then, write $|Q(n)+Q(m-n)-k|\leq 1/2$ as $|Q(2n-m)+Q(m)-2k|\leq 1$ and letting $2n\in m+G_l$ where $l=2k-Q(m)$ and $G_l=\{a\in \Z^2 :|Q(a)-l|\leq 1\}$, we get
  \begin{eqnarray}
  \|e^{it\Delta}f\|_{L^4_tL^4_x}^2 &\les& \Big{[}\sum_{m \in \Z^2}\big{[}\sum_{l \in \Z}|\sum_{2n\in m+G_l}\hat{f}(n)\hat{f}(m-n)|^2\big{]}\Big{]}^{1/2}\non\\
  &\les& \Big{[}\sum_{m \in \Z^2}\big{[}\sum_{l \in \Z}|G_l|^{1/2}|\sum_{2n\in m+G_l}\hat{f}(n)^2\hat{f}(m-n)^2|^{1/2}\big{]}^2\Big{]}^{1/2},\non
   \end{eqnarray}

since $G_l=\{a\in \Z^2 :|Q(a)|\leq 1+l\}-\{a\in \Z^2 :|Q(a)|< 1-l\}$, using Theorem \ref{hux}, we get $$|G_l|\les l^{s_0+},$$ and hence, using $l\les N^2$, we obtain,

  \begin{eqnarray}
  \|e^{it\Delta}f\|_{L^4_tL^4_x}^2 &\les& N^{s_0+}\Big{[}\sum_{m \in \Z^2}\big{[}\sum_{l \in \Z}|\sum_{2n\in m+G_l}\hat{f}(n)^2\hat{f}(m-n)^2|\big{]}\Big{]}^{1/2},\non\\
   &\les& N^{s_0+}\big{[}\sum_{m \in \Z^2}|\sum_{n \in \Z^2}\hat{f}(n)^2\hat{f}(m-n)^2|\big{]}^{1/2},\non\\
   &\les& N^{s_0+}\|f\|_{L_x^2}^2.\non
  \end{eqnarray}
   
   Therefore, the result.
   \end{proof}

\subsubsection{Proof of Theorem \ref{local}}   

   As mentioned above, to prove local well-posedness we use Bourgain spaces. Since Bourgain spaces behave nice under linear evolution, what we need to show is that the nonlinear part of the Duhamel formula also behaves as nice. For that we need:
   
   \begin{prop}\label{duhamel}
   
   For $b,b'$ such that $0\leq b+b'<1$, $0\leq b'<1/2$, then we have $$\|\int_0^t e^{i\Delta(t-\tau)}f(\tau)d\tau\|_{X_T^{s,b}}\les T^{1-b-b'}\|f\|_{X_T^{s,-b'}},$$ for $T\in [0,1]$.
   
   \end{prop} 

   The proof of this result is standard, see \cite{[gin]}. Hence, to be able to use the Banach Fixed Point Theorem, we have to control the right hand side of the inequality in the appropriate $X^{s,b}$ space. And since our nonlinearity is cubic, that means have to show a trilinear estimate:
   
   \begin{prop}\label{trilinear}
   
   For $s> s_0$, there exists $b, b'$ satisfying the conditions of Proposition \ref{duhamel}, such that, $$\|u_1 u_2 \overline{u_3}\|_{X_T^{s,-b'}}\les \|u_1\|_{X_T^{s,b}}\|u_2\|_{X_T^{s,b}}\|u_3\|_{X_T^{s,b}}.$$
   
   \end{prop}
   
   Hence, it is clear that once we prove Proposition \ref{trilinear}, Theorem \ref{local}, i.e. the local well-posedness will follow.
   
    We will prove Proposition \ref{trilinear} using the duality argument $\|u_1 u_2 \overline{u_3}\|_{X_T^{s,-b'}}=\sup_{(\|u_4\|_{X_T^{-s,b'}}=1)}\int\int u_1 u_2 \overline{u_3}\overline{u_4}dxdt$. Hence, to prove Proposition \ref{trilinear}, we will bound the integral on the right hand side of this equality by $\|u_1\|_{X_T^{s,b}}\|u_2\|_{X_T^{s,b}}\|u_3\|_{X_T^{s,b}}.$ 
   
   \begin{proof}{(Proof of Proposition \ref{trilinear})}

    We first need a fundamental bilinear Strichartz estimate:
    
    \begin{lemma}\label{bistr}

    If $u_1, u_2\in L_x^2$ s.t. $supp(\wh{u_1})\in B(0,N_1)$ and $supp(\wh{u_2})\in B(0,N_2)$ with $N_1\leq N_2$. Then we have $$\|e^{it\Delta}u_1 e^{it\Delta}u_2\|_{L_x^2}\les N_1^{s_0+}\|u_1\|_{L_x^2} \|u_2\|_{L_x^2}.$$  
    
    \end{lemma}
    
    \begin{proof}

    Let $P_{I}$ be the partition of $\Z^2$ into boxes $I$ of size $N_1$. We can decompose $u_2$ as $u_2=\sum_I u_2^{(I)}=\sum_I P_{I} u_2$ and by almost orthogonality, and that $e^{it\Delta}u_1 e^{it\Delta}u_2^{(I)}=P_{5I}(e^{it\Delta}u_1 e^{it\Delta}u_2^{(I)})$, which follows from the convolution property, we have,
    
    \begin{eqnarray}
      \|e^{it\Delta}u_1 e^{it\Delta}u_2\|_{L^2_tL^2_x} &\leq& \|\sum_I e^{it\Delta}u_1 e^{it\Delta}u_2^{(I)}\|_{L^2_tL^2_x}\non\\
    &\les& (\sum_I \|e^{it\Delta}u_1 e^{it\Delta}u_2^{(I)}\|_{L^2_tL^2_x}^2)^{1/2}\non\\
    &\les& (\sum_I \|e^{it\Delta}u_1\|_{L^4_tL^4_x}^2 \|e^{it\Delta}u_2^{(I)}\|_{L^4_tL^4_x}^2)^{1/2}\non\\
    &\les& N_1^{\frac{s_0+}{2}}\|u_1\|_{L^2_x}(\sum_I \|u_2^{(I)}\|_{L^2_x}^2)^{1/2}\non\\
    &\les&  N_1^{s_0+}\|u_1\|_{L^2_x}\|u_2\|_{L^2_x}.\non
   \end{eqnarray}
    
    \end{proof}
    
    Using this bilinear Strichartz estimate we can also prove:
   \begin{lemma}\label{bistr-xsb}
    Let any $u_1, u_2\in X^{0,b}$ such that the Fourier transforms of $u_1$ and $u_2$ are supported in $[N_1, 2N_1]$ and $[N_2, 2N_2]$ respectively with $N_1\leq N_2$. Then we have, 
    \begin{equation}\label{inter-1/2}
    \|u_1u_2\|_{L_t^2L_x^2}\les N_1^{s_0+}\|u_1\|_{X^{0,b}}\|u_2\|_{X^{0,b}}.
    \end{equation}
   \end{lemma}
   \begin{proof}
   We take $\tilde{u_i}(t,x)$ on $\mathbb{R}\times \mathbb{T}_{\theta}^2$ such that $\tilde{u_i}(t,x)=u_i(t,x)$ for $(t,x)\in \mathbb{T}\times \mathbb{T}_{\theta}^2$. With an abuse of notation we will call $\tilde{u_i}=u_i$. For $n\in \Z^2$, $Q(n)$ being the symbol of Laplacian we have,
   \begin{eqnarray}
    u_i(t,x)&=&\frac{1}{(2\pi)^3}\int\sum_{n\in\Z^2}\widehat{u_i}(\tau,n)e^{it\tau}e^{ix.n} d\tau\non\\
    &=&\frac{1}{(2\pi)^3}\int\sum_{n\in\Z^2} \widehat{u_i}(\tau,n)e^{it\tau}e^{ix.n}e^{itQ(n)}e^{-itQ(n)} d\tau\non\\
    &=&\frac{1}{(2\pi)^3}\int\sum_{n\in\Z^2} \widehat{u_i}(\tau,n)e^{it(\tau+Q(n))}e^{ix.n}e^{-itQ(n)}d\tau\non\\
    &=&\frac{1}{(2\pi)^3}\int\sum_{n\in\Z^2} \widehat{u_i}(\tau-Q(n),n)e^{it\tau}e^{ix.n}e^{-itQ(n)} d\tau\non\\
    &=&\frac{1}{(2\pi)}\int e^{it\Delta}\widehat{v_i}(\tau,x)e^{it\tau} d\tau,\non
   \end{eqnarray}
   by the definition of linear propagator, where $\widehat{v_i}(\tau,x)=\frac{1}{(2\pi)^2}\sum_{n\in\Z^2} \widehat{u_i}(\tau-Q(n),n)e^{ix.n}$. Thus,
   \begin{eqnarray}
   \int \int u_1u_2dxdt&=&\int e^{it(\tau_1+\tau_2)}e^{it\Delta}\widehat{v_1}(\tau_1,x)e^{it\Delta}\widehat{v_2}(\tau_2,x) d\tau_1 d\tau_2dtdx\non\\
   &=&\int e^{it(\tau_1+\tau_2)}e^{it\Delta}\widehat{v_1}(\tau_1,x)e^{it\Delta}\widehat{v_2}(\tau_2,x)dtdx d\tau_1 d\tau_2\non\\
   &\les& \int N_1^{s_0+}\|\widehat{v_1}\|_{L_x^2}\|\widehat{v_2}\|_{L_x^2}d\tau_1 d\tau_2\non\\
   &=& N_1^{s_0+}\int\|\widehat{v_1}\|_{L_x^2}d\tau_1\int\|\widehat{v_2}\|_{L_x^2}d\tau_2,\non
   \end{eqnarray}
   and for each $i$ we use,
   \begin{eqnarray}
   \int\|\widehat{v_i}\|_{L_x^2}d\tau_i&=&\int\frac{\langle\tau_i\rangle^b}{\langle \tau_i\rangle^b}\|\widehat{v_i}\|_{L_x^2}d\tau_i\non\\
   &\les&(\int\langle\tau_i \rangle^{2b}\|\widehat{v_i}\|_{L_x^2}^2d\tau_i)^{1/2}\non\\
   &=&\|u_i\|_{X^{0,b}}.\non
   \end{eqnarray}
   and the result follows by taking the infimum of such $u_i$'s.
   \end{proof}       
   Also, using embedding $X^{0,(1/4)+}\subset L^4_tL^2_x$, which is obtained by interpolation between $X^{0,b}\subset L^\infty_tL^2_x $ for $b>1/2$ and $X^{0,0}=L^2_tL^2_x$, we see that,
   \begin{eqnarray}
      \|u_1u_2\|_{L^2_tL^2_x}&\leq& \|u_1\|_{L^4_tL^{\infty}_x} \|u_2\|_{L^4_tL^2_x}\non\\
    &\les& N_1\|u_1\|_{L^4_tL^2_x}\|u_2\|_{L^4_tL^2_x}\non\\
    &\les&  N_1\|u_1\|_{X^{0,(1/4)+}}\|u_2\|_{X^{0,(1/4)+}}.\label{inter-1/4}
   \end{eqnarray}
   
   Now we can prove a crude interpolation between \eqref{inter-1/2} and \eqref{inter-1/4} and get:

    \begin{lemma}
    Let $u_1, u_2\in X^{0,b}$ such that the Fourier transforms of $u_1$ and $u_2$ are supported in $[N_1, 2N_1]$ and $[N_2, 2N_2]$ respectively with $N_1\leq N_2$. Then for $s>s_0$ there exists $b'< 1/2$ such that $$\|u_1u_2\|_{L_t^2L_x^2}\les N_1^s\|u_1\|_{X^{0,b'}}\|u_2\|_{X^{0,b'}}.$$
    \end{lemma}
    \begin{proof}
    We have,
    \begin{equation}\label{est1/4}
    \|u_1u_2\|_{L_t^2L_x^2}\les N_1\|u_1\|_{X^{0,(1/4)+}}\|u_2\|_{X^{0,(1/4)+}},
    \end{equation}
    and we also have,
    \begin{equation}\label{estb}
    \|u_1u_2\|_{L_t^2L_x^2}\les N_1^{s_0+}\|u_1\|_{X^{0,b}}\|u_2\|_{X^{0,b}}.
    \end{equation}
    
     Note that $\eqref{est1/4}\leq N_1\|u_1\|_{X^{0,b}}\|u_2\|_{X^{0,(1/4)+}}$. Then for fixed $u_1$ interpolating between this result and \eqref{estb}, we get,
     \begin{equation}\label{estinter}
     \|u_1u_2\|_{L_t^2L_x^2}\les N_1^{\tilde{s}}\|u_1\|_{X^{0,b}}\|u_2\|_{X^{0,b'}},
     \end{equation}
     for some $\tilde{s}\in [s_0,s)$ and $b'<1/2$. Also note that $\eqref{est1/4}\leq N_1\|u_1\|_{X^{0,(1/4)+}}\|u_2\|_{X^{0,b'}}$. Thus, for fixed $u_2$, interpolating between this result and \eqref{estinter} we obtain, $$\|u_1u_2\|_{L_t^2L_x^2}\les N_1^s\|u_1\|_{X^{0,b'}}\|u_2\|_{X^{0,b'}}.$$
    \end{proof}

   Thus we get that, if $u_i\in X^{0,b}$, $i\in \{1,2,3,4\}$ are functions s.t. their space Fourier transforms are supported in $[N_i, 2N_i]$ respectively with $N_1\leq N_2\leq N_3\leq N_4$ and $s>s_0$, there exists $b'< 1/2$ such that 
  \begin{eqnarray}\label{interpolation}
  \int \int u_1\overline{u_2}u_3\overline{u_4}dxdt&\les&\|u_1u_3\|_{L^2_tL^2_x} \|u_2u_4\|_{L^2_tL^2_x}\non\\
  &\leq& (N_1 N_2)^s \|u_1\|_{X^{0,b'}} \|u_2\|_{X^{0,b'}} \|u_3\|_{X^{0,b'}} \|u_4\|_{X^{0,b'}}.\non
  \end{eqnarray}
  
  We are almost ready to finish the proof of the proposition. All we need now is to guarantee the existence of $b,b'$ which satisfy the conditions of Proposition \ref{duhamel}. But for that we need better estimates on the restrictions of functions on the eigenspaces of the Laplacian. Let $\mbox{\Large$\wp$}_k$ be the projection onto the $e_k$, the eigenspace of Laplacian corresponding to the eigenvalue $\mu_k$. Also for each $e_k$ we see that $\mu_ke_k=-\Delta e_k$ which implies that $$\mu_k\widehat{e_k}(m_1,m_2)=((\theta_1m_1)^2+(\theta_2m_2)^2)\widehat{e_k}(m_1,m_2),$$ hence, $\widehat{e_k}$'s are supported on $\mu_k=(\theta_1m_1)^2+(\theta_2m_2)^2=Q(m_1,m_2)$ i.e. they are supported on $\mu_k^s=|Q(m_1,m_2)|^{2s}$. This gives that $$\mu_k^se_k=(\sqrt{-\Delta})^{2s} e_k.$$
   Since $\|\mbox{\Large$\wp$}_ku\|_{L_x^2}\leq\|u\|_{L_x^2}$ and that $e_k$'s form an orthonormal basis, we can define Sobolev space $H^s$, with the norm, $\|u\|_{H^s}^2=\sum_k \langle \mu_k\rangle^s\|\mbox{\Large$\wp$}_k u\|_{L^2}^2$ which we will be using later in the paper.
   
   Also we see that since $\theta_1/\theta_2$ is irrational, if $\mu_k=(\theta_1m_1)^2+(\theta_2m_2)^2=(\theta_1n_1)^2+(\theta_2n_2)^2$ we have $(m_1,m_2)=(\pm n_1, \pm n_2)$, which means, for any such $\mu_k$, we have four eigenfunctions $(e^{ix.(\pm m_1, \pm m_2)})$ and $\mbox{\Large$\wp$}_k u$ is the the restriction of $u$ to the eigenspace generated by these eigenfunctions.
   
   Now if we consider the integrals of the form $$A=\int e^{ix.(m_1, m_2)}e^{ix.(n_1,n_2)}e^{ix.(j_1,j_2)}e^{ix.(l_1,l_2)}a_ma_na_ja_l dx,$$ we see that $A=0$ if $(m_1+n_1+j_1+l_1,m_2+n_2+j_2+l_2)\neq 0$. Thus if $|m_1|>4max(|n_1|,|j_1|,|l_1|)$ or $|m_2|>4max(|n_2|,|j_2|,|l_2|)$, then $A=0$. This says that, if $\mu_{k_4}^{1/2}>8\mu_{k_i}^{1/2}$ for $i=\{1,2,3\},$ then
   $$\int \mbox{\Large$\wp$}_{k_1}(u_1)\mbox{\Large$\wp$}_{k_2}(u_2)\mbox{\Large$\wp$}_{k_3}(u_3)\mbox{\Large$\wp$}_{k_4}(u_4)dx=0,$$
  and we will use this observation in our estimates.
     
    Now we can show the existence of $1/4<b'<1/2<b$ s.t. for every $s>s_0$, $$\|u_1u_2\overline{u_3}\|_{X^{s,-b'}_T}\les \|u_1\|_{X^{s,b}_T} \|u_2\|_{X^{s,b}_T} \|u_3\|_{X^{s,b}_T},$$ which will finish the proof of local well-posedness.

   As we mentioned before, we will bound, $|\int u_1u_2\overline{u_3}\overline{u_4}dxdt|$. To do so, it is enough to bound this integral for $u_i=\mbox{\large$\chi$}^{\sqrt{-\Delta}}_{N_i}u_i$, where $N_i$ is a dyadic integer. Let without loss of generality that $N_1\leq N_2\leq N_3$ and let $s'\in (s_0,s)$ then for the range $N_4\leq 8N_3$,
   
   \begin{eqnarray}
   |\int u_1u_2\overline{u_3}\overline{u_4}dxdt| &\les& (N_1N_2)^{s'}\|u_1\|_{X^{0,b'}_T}\|u_2\|_{X^{0,b'}_T}\|u_3\|_{X^{0,b'}_T}\|u_4\|_{X^{0,b'}_T}\non\\
   &=& (N_1N_2)^{s'-s}(N_4/N_3)^{s}N_1^s\|u_1\|_{X^{0,b'}_T}N_2^s\|u_2\|_{X^{0,b'}_T}N_3^s\|u_3\|_{X^{0,b'}_T}N_4^{-s}\|u_4\|_{X^{0,b'}_T}\non\\
   &\les&(N_1N_2)^{s'-s}(N_4/N_3)^{s}\|u_1\|_{X^{s,b'}_T}\|u_2\|_{X^{s,b'}_T}\|u_3\|_{X^{s,b'}_T}\|u_4\|_{X^{-s,b'}_T}.\non
   \end{eqnarray}

   Hence, for the range of the frequencies, write $N_4=2^nN_3$ for $n\leq3$ and then we have 
   \begin{eqnarray}|\int u_1u_2\overline{u_3}\overline{u_4}dxdt|&\les& (N_1N_2)^{s'-s}2^{ns}\|u_1\|_{X^{s,b'}_T}\|u_2\|_{X^{s,b'}_T}\|u_3\|_{X^{s,b'}_T}\|u_4\|_{X^{-s,b'}_T}\non\\
   &=&(N_1N_2)^{s'-s}2^{ns}\|u_1\|_{X^{s,b'}_T}\|u_2\|_{X^{s,b'}_T}\|\mbox{\large$\chi$}^{\sqrt{-\Delta}}_{N_42^{-n}}u_3\|_{X^{s,b'}_T}\|\mbox{\large$\chi$}^{\sqrt{-\Delta}}_{N_4}u_4\|_{X^{-s,b'}_T},\non
   \end{eqnarray}
    and summing in $N_1, N_2, N_4$ and $n\leq 3$ we get
  \begin{eqnarray}
  |\int u_1u_2\overline{u_3}\overline{u_4}dxdt|&\les& \sum_{n\leq 3}\sum_{N_1}\sum_{N_2} (N_1N_2)^{s'-s}2^{ns}\|u_1\|_{X^{s,b'}_T}\|u_2\|_{X^{s,b'}_T}\non\\
  &&\qquad\qquad (\sum_{N_4}\|\mbox{\large$\chi$}^{\sqrt{-\Delta}}_{N_42^{-n}}u_3\|_{X^{s,b'}_T}^2)^{1/2}(\sum_{N_4}\|\mbox{\large$\chi$}^{\sqrt{-\Delta}}_{N_4}u_4\|_{X^{-s,b'}_T}^2)^{1/2}\non\\
  &\les& \|u_1\|_{X^{s,b'}_T} \|u_2\|_{X^{s,b'}_T} \|u_3\|_{X^{s,b'}_T}\|u_4\|_{X^{-s,b'}_T},\non
  \end{eqnarray}
   
    where we used the $H^s$-orthogonality of the operators $\mbox{\large$\chi$}^{\sqrt{-\Delta}}_{N}$, for $N$ dyadic integers. And for the range $8N_3\leq N_4$, we use the observation above and get
   $|\int u_1u_2\overline{u_3}\overline{u_4}dxdt|=0$. Thus the Proposition \eqref{trilinear} follows and hence Theorem \ref{local}.
   \end{proof}
   
   \subsubsection{Proof of Theorem \ref{growth}}
   
   The proof of the theorem will mainly follow Bourgain's arguments in \cite{[bou-growth]}, i.e. we will use Lemma \ref{poly-est}. For that, first we need to observe that for $s>1$, in the proof of Proposition \ref{trilinear} if we take $u_1=u_2=u_3=u$ and $s'=1-$, redoing the calculations we get $$\|u|u|^2\|_{X^{s,-b'}_T}\les \|u\|_{X^{s,b}_T} \|u\|_{X^{1,b}_T}^2.$$ This says, we can choose the local well-posedness interval depending only on $\|u(0)\|_{H^1}$. Thus we can find $T_0>0$ such that, for any time $\tau\geq0$ the solution exists for $t\in [\tau,\tau+T_0]$. Now we need to find $r \in (0,1)$ such that for any $t\in [\tau,\tau+T_0]$, 
   \begin{equation}\label{bound}\|u(t)\|_{H_x^s}\leq \|u(\tau)\|_{H_x^s}+C\|u(\tau)\|_{H_x^s}^{1-r}.
   \end{equation}
    Since $L_x^2$-norm of the solution is conserved, it is enough to show this estimate in $\dot{H}_x^s$. Without loss of generality we can take $\tau=0$. Since $$\|u(t)\|_{\dot{H}_x^s}^2- \|u(0)\|_{\dot{H}_x^s}^2=\int_{0}^t\frac{d}{dt'}\|u(t')\|_{\dot{H}_x^s}^2dt',$$ if we show that, for $t\in [0,T_0]$ we have, $$\int_{0}^t\frac{d}{dt'}\|u(t')\|_{\dot{H}_x^s}^2dt'\les \|u\|_{H^s}\|u\|_{H^{s-\sigma}},$$ for some $s-1\geq\sigma>0$, writing $H^{s-\sigma}$ as the interpolation space between $H^1$ and $H^s$ we will obtain, $$\|u(t)\|_{\dot{H}_x^s}^2- \|u(0)\|_{\dot{H}_x^s}^2\les\|u\|_{H^s}\|u\|_{H^s}^{\frac{(s-\sigma-1)}{(s-1)}},$$ which is the estimate we want, where the implicit constant also depends on $\|u\|_{H^{1}}$. For $\sigma>s-1$, $H^{1}$ embeds into $H^{s-\sigma}$ and the result becomes obvious. Now assume $s\in \N$,
   
   \begin{eqnarray}
   \int_{0}^t\frac{d}{dt'}\|u(t')\|_{\dot{H}_x^s}^2dt'&=&\int_{0}^t\frac{d}{dt'}\|D^s u(t')\|_{L_x^2}^2 dt'\non\\
  &=& 2Re \int_{0}^t\int_{T_\theta^2}\frac{d}{dt'}D^s \overline{u}(t')D^s u(t')dxdt',\non
   \end{eqnarray}
   
   and using the expression for $u_t$, we get,
   
   \begin{eqnarray}
   \int_{0}^t\frac{d}{dt'}\|u(t')\|_{\dot{H}_x^s}^2dt'&=& 2Im \int_{0}^t\int_{T_\theta^2}D^s \overline{u}D^s (|u|^2u)dxdt' \non\\
   &=& 4Im \int_{0}^t\int_{T_\theta^2}|D^s u|^2 |u|^2 dxdt'+ 2Im\int_{0}^t\int_{T_\theta^2}(D^s \overline{u})^2u^2 dxdt'\non\\
   &&\qquad\qquad\qquad+2Im\int_{0}^t\int_{T_\theta^2}\sum_{|\alpha|=s \atop \alpha_i\neq s}D^s \overline{u}\partial^{\alpha_1}\overline{u}\partial^{\alpha_2}u\partial^{\alpha_3}u dxdt'\non\\
   &=& 2Im\int_{0}^t\int_{T_\theta^2}(D^s \overline{u})^2 u^2 dxdt'+2Im\int_{0}^t\int_{T_\theta^2}\sum_{|\alpha|=s \atop \alpha_i\neq s}D^s \overline{u}\partial^{\alpha_1}\overline{u}\partial^{\alpha_2}u\partial^{\alpha_3}u dxdt'\non\\
   &=& I+II.\non
   \end{eqnarray}

   Second term is easier to estimate. For any multiindex $|\alpha|=s$ such that $\alpha_i\neq s$ for any $i$, using duality and Proposition \eqref{trilinear}, we have,
   
   \begin{eqnarray}
   II&\les& \|D^s u\|_{X_{T_0}^{-s_0-,b}}\|\partial^{\alpha_1}u\partial^{\alpha_2}\overline{u}\partial^{\alpha_3}\overline{u}\|_{X_{T_0}^{s_0+,-b}}\non\\
   &\les& \|u\|_{X_{T_0}^{s-s_0-,b}}\|\partial^{\alpha_1}u\|_{X_{T_0}^{s_0+,b}}\|\partial^{\alpha_2}u\|_{X_{T_0}^{s_0+,b}}\|\partial^{\alpha_3}u\|_{X_{T_0}^{s_0+,b}}\non\\
   &\les&\|u\|_{X_{T_0}^{s-s_0-,b}}\|u\|_{X_{T_0}^{s_0+{\alpha_1}+,b}}\|u\|_{X_{T_0}^{s_0+{\alpha_2}+,b}}\|u\|_{X_{T_0}^{s_0+{\alpha_3}+,b}}.\non
   \end{eqnarray}

   For $1\leq \alpha_i \leq s-1$, using interpolation and the fact that $\|u\|_{X_{T_0}^{1,b}}$ is bounded, which follows from the local theory, we get,
   \begin{equation}
   II\les \|u\|_{X_{T_0}^{s,b}}^{\frac{s-s_0-1-}{s-1}}\|u\|_{X_{T_0}^{s,b}}^{\frac{s_0+\alpha_1-1+}{s-1}}\|u\|_{X_{T_0}^{s,b}}^{\frac{s_0+\alpha_2-1+}{s-1}}\|u\|_{X_{T_0}^{s,b}}^{\frac{s_0+\alpha_3-1+}{s-1}}.\non
   \end{equation}
    If for some $i\in\{1,2,3\}$, $\alpha_i= 0$, say $\alpha_3=0$, using $\|u\|_{X_{T_0}^{s_0+,b}}\leq\|u\|_{X_{T_0}^{1,b}}$ we get,
    \begin{equation}
    II\les \|u\|_{X_{T_0}^{s,b}}^{\frac{s-s_0-1-}{s-1}}\|u\|_{X_{T_0}^{s,b}}^{\frac{s_0+\alpha_1-1+}{s-1}}\|u\|_{X_{T_0}^{s,b}}^{\frac{s_0+\alpha_2-1+}{s-1}}.
    \end{equation}

    Thus we get the desired bound using $\|u\|_{X_{T_0}^{s,b}}\les\|u(0)\|_{H^s}$ in the local well-posedness interval.
    
   The term $I$ is harder to deal with since the highest order derivatives acts on $\overline{u}$. The main problem here is that, because of the term $(D^s\overline{u})^2$ in the integrand we expect to have a bound of the form $II\les \|u\|_{X_{T_0}^{s,b}}^2$ which is not useful. To remedy that problem, we will try to get 
   \begin{equation}\label{lastest}
    II\les \|u\|_{X_{T_0}^{s,b}}\|u\|_{X_{T_0}^{s-\sigma,b}}\|u\|_{X_{T_0}^{1,b}}\|u\|_{X_{T_0}^{1,b}},
   \end{equation}
    for some $\sigma >0$ to be determined. In the following estimates we will mainly follow Zhong's arguments in \cite{[zho]}.
   
   Let $D^s\overline{u}=u_1=u_2$, $u_3=u_4=u$, and $u_i= \sum_j \mbox{\large$\chi$}^{\sqrt{-\Delta}}_{N_{(i,j)}}u_i=\sum_j u_i^j$ where $N_{(i,j)}$'s are dyadic integers. Then $$|II|\leq \sum_N |II(N)|=\sum_N |\int_{\mathbb T_\theta^2}\int_{[0,T_0]}u_1^j u_2^k u_3^m u_4^n|.$$
   
   Since we need to get an estimate of the form \eqref{lastest}, we should gain some derivative in the estimate of $II(N)$. For the terms $N_{(1,j)}>8( N_{(2,k)}+N_{(3,m)}+N_{(4,n)})$, we again see that $II(N)=0$. Hence we have to focus on the terms where $N_{(1,j)} <8( N_{(2,k)}+N_{(3,m)}+N_{(4,n)})$.
   
    Assume $N_{(1,j)}<8( N_{(2,k)}+N_{(3,m)}+N_{(4,n)})$, and thus, $N_{(1,j)} \les max (N_{(2,k)},N_{(3,m)},N_{(4,n)})$. Since $u_2$ has full $s$-derivative, we will estimate $II$ using the interaction between frequency projections of $u_2$ with $u_3$ and $u_4$. We consider two cases; $N_{(2,k)} < 4 N_{(3,m)}$ or $N_{(2,k)} < 4 N_{(4,n)}$ as case one and $N_{(2,k)} \geq 4 N_{(3,m)}$ and $N_{(2,k)} \geq 4 N_{(4,n)}$ as case two. 
    
    \textbf{Case 1:} $N_{(2,k)} < 4 N_{(3,m)}$ or $N_{(2,k)} < 4 N_{(4,n)}$. This case gives a control over the $N_{(2,k)}$ term and is easier to handle. Without loss of generality we can assume, $N_{(2,k)} < 4 N_{(3,m)}$ and $N_{(4,n)}\ges 1$. Hence by Lemma \ref{bistr},
   \begin{eqnarray}
   II(N)&\leq& \|u_1^ju_3^m\|_{L^2_tL^2_x([0,T_0])}\|u_2^k u_4^n\|_{L^2_tL^2_x([0,T_0])}\non\\
   &\les& min(N_{(1,j)},N_{(3,m)})^{s_0+} min(N_{(2,k)},N_{(4,n)})^{s_0+}\|u_1^j\|_{X^{0,b}_{T_0}}\|u_2^k\|_{X^{0,b}_{T_0}}\|u_3^m\|_{X^{0,b}_{T_0}}\|u_4^n\|_{X^{0,b}_{T_0}}\non\\
   &\les& (N_{(3,m)}N_{(4,n)})^{s_0+}\|u_1^j\|_{X^{0,b}_{T_0}}\|u_2^k\|_{X^{0,b}_{T_0}}\|u_3^m\|_{X^{0,b}_{T_0}}\|u_4^n\|_{X^{0,b}_{T_0}}\non\\
   &\les& (N_{(3,m)}N_{(4,n)})^{s_0-1+}\|u_1^j\|_{X^{0,b}_{T_0}}\|u_2^k\|_{X^{0,b}_{T_0}}\|u_3^m\|_{X^{1,b}_{T_0}}\|u_4^n\|_{X^{1,b}_{T_0}}\non\\
   &\les&(N_{(1,j)}N_{(2,k)}N_{(3,m)}N_{(4,n)})^-N_{(1,j)}^+N_{(2,k)}^+(N_{(3,m)}N_{(4,n)})^{(s_0-1)+}\non\\
   &&\qquad\qquad\qquad\times\|u_1^j\|_{X^{0,b}_{T_0}}\|u_2^k\|_{X^{0,b}_{T_0}}\|u_3^m\|_{X^{1,b}_{T_0}}\|u_4^n\|_{X^{1,b}_{T_0}},\non\\
    &\les&(N_{(1,j)}N_{(2,k)}N_{(3,m)}N_{(4,n)})^-N_{(2,k)}^+(N_{(3,m)}N_{(4,n)})^{(s_0-1)+}\non\\
    &&\qquad\qquad\qquad\times\|u_1^j\|_{X^{0,b}_{T_0}}\|u_2^k\|_{X^{0,b}_{T_0}}\|u_3^m\|_{X^{1,b}_{T_0}}\|u_4^n\|_{X^{1,b}_{T_0}},\non
   \end{eqnarray}
   and given $N_{(2,k)} < 4 N_{(3,m)}$ we see, 
   \begin{eqnarray}
   II(N)&\les& (N_{(1,j)}N_{(2,k)}N_{(3,m)}N_{(4,n)})^-N_{2,k}^{((s_0-1))+}\|u_1^j\|_{X^{0,b}_{T_0}}\|u_2^k\|_{X^{0,b}_{T_0}}\|u_3^m\|_{X^{1,b}_{T_0}}\|u_4^n\|_{X^{1,b}_{T_0}}\non\\
   &\les& (N_{(1,j)}N_{(2,k)}N_{(3,m)}N_{(4,n)})^- \|u_1^j\|_{X^{0,b}_{T_0}}\|u_2^k\|_{X^{((s_0-1))+,b}_{T_0}}\|u_3^m\|_{X^{1,b}_{T_0}}\|u_4^n\|_{X^{1,b}_{T_0}},\non
   \end{eqnarray}
   which gives the desired result for $\sigma=(1-s_0)+$.
   
   \textbf{Case 2:} $N_{(2,k)} \geq 4 N_{(3,m)}$ and $N_{(2,k)} \geq 4 N_{(4,n)}$. Recall that since $N_{(3,m)} \leq 1/4 N_{(2,k)}$ and $N_{(4,m)} \leq 1/4 N_{(2,k)}$, we have, $N_{(1,j)}\leq 12 N_{(2,k)}$, in which case we define,
   \begin{equation}\label{decomp} u_i^{j,j'}= \sum_{(N_{(i,j)}\leq \langle\mu_k\rangle^{1/2}\leq 2N_{(i,j)})} \quad\int\limits_{(L_{(i,j')}\leq \langle\mu_k+\tau\rangle\leq 2L_{(i,j')})}e^{it\tau}\widehat{\mbox{\Large$\wp$}_k u_i}(\tau)d\tau, 
   \end{equation}
   for $L_{(i,j')}$ dyadic integers, where $\mu_k$'s being the eigenvalues of Laplacian and $\mbox{\Large$\wp$}_k u$ being the projection of $u$ on the eigenspace corresponding to $\mu_k$
   Then we have, 
   \begin{eqnarray}
   II(N)&\leq& \sum_L |\int u_1^{j,j'}u_2^{k,k'}u_3^{m,m'}u_4^{n,n'}dxdt|\non\\
   &\leq& \sum_L |\int\limits_{\{\sum_{i=1}^4 \tau_i=0\}}\widehat{u_1}^{j,j'}\widehat{u_2}^{k,k'}\widehat{u_3}^{m,m'}\widehat{u_4}^{n,n'}dxd\tau|\non\\
   &=& \sum_L II(N,L),\non
   \end{eqnarray}
   where the Fourier transform is with respect to time only and $L=(L_{(1,j')},L_{(2,k')},L_{(3,m')},L_{(4,n')})$. Thus we need to estimate $II(N,L)$ for each $L$. Since we are concerned with the derivative gain for $u_2$ in the estimate of $II(N)$, for each $L$ we need to use the relation between $L$-terms and $N_{(2,k)}$. For that we consider two cases again; $|\tau_3| \leq 1/3 N_{(2,k)}^2$ and $|\tau_4| \leq 1/3 N_{(2,k)}^2$ as case one, and $|\tau_3| > 1/3 N_{(2,k)}^2$ or $|\tau_4| > 1/3 N_{(2,k)}^2$ as case two. 
   
   \textbf{Case 1:} $|\tau_3| \leq 1/3 N_{(2,k)}^2$ and $|\tau_4| \leq 1/3 N_{(2,k)}^2$. For this case, using the decomposition \eqref{decomp} we get,
   \begin{eqnarray}
   |\mu_{k_1}+\tau_1|+|\mu_{k_2}+\tau_2|&\geq& |\mu_{k_1}+\tau_1+\mu_{k_2}+\tau_2|\non\\
   &\geq& |\mu_{k_1}+\mu_{k_2}|-|\tau_1+\tau_2|\non\\
   &\geq& \mu_{k_2}-|\tau_1+\tau_2|\non\\
   &=& \mu_{k_2}-|\tau_3+\tau_4|\non\\
   &\ges& N_{(2,k)}^2,\non
   \end{eqnarray}
   
   and thus, $L_{(1,j')}+L_{(2,k')}\ges N_{(2,k)}^2$, which also gives $max\{L_{(1,j')},L_{(2,k')}\}\ges N_{(2,k)}^2$. So we get,
   \begin{equation}
   II(N,L) \leq \|u_1^{j,j'}\|_{L_t^4L_x^2}\|u_2^{k,k'}\|_{L_t^4L_x^2}\|u_3^{m,m'}\|_{L_t^4L_x^\infty}\|u_4^{n,n'}\|_{L_t^4L_x^\infty},\non
   \end{equation}
   then by Sobolev embedding we get,
   \begin{equation}
   II(N,L)\les (N_{(3,m)}N_{(4,n)})\|u_1^{j,j'}\|_{L_t^4L_x^2}\|u_2^{k,k'}\|_{L_t^4L_x^2}\|u_3^{m,m'}\|_{L_t^4L_x^2}\|u_4^{n,n'}\|_{L_t^4L_x^2},\non
   \end{equation}
   and using $X^{0,1/4+}\subset L_t^4L_x^2$, we obtain,
   \begin{eqnarray}
   II(N,L)&\les& (N_{(3,m)}N_{(4,n)})\|u_1^{j,j'}\|_{X^{0,1/4+}}\|u_2^{k,k'}\|_{X^{0,1/4+}}\|u_3^{m,m'}\|_{X^{0,1/4+}}\|u_4^{n,n'}\|_{X^{0,1/4+}}\non\\
   &\les& \|u_1^{j,j'}\|_{X^{0,1/4+}}\|u_2^{k,k'}\|_{X^{0,1/4+}}\|u_3^{m,m'}\|_{X^{1,1/4+}}\|u_4^{n,n'}\|_{X^{1,1/4+}}\non\\
   &\les& \frac{1}{(L_{(1,j')}L_{(2,k')}L_{(3,m')}L_{(4,n')})^{b-1/4-}}\|u_1^{j,j'}\|_{X^{0,b}}\|u_2^{k,k'}\|_{X^{0,b}}\|u_3^{m,m'}\|_{X^{1,b}}\|u_4^{n,n'}\|_{X^{1,b}}\non\\
   &\les& \frac{N_{2,k}^{2(1/4-b+)}}{L_{(1,j')}^+L_{(2,k')}^+(L_{(3,m')}L_{(4,n')})^{b-1/4-}}\|u_1^{j,j'}\|_{X^{0,b}}\|u_2^{k,k'}\|_{X^{0,b}}\|u_3^{m,m'}\|_{X^{1,b}}\|u_4^{n,n'}\|_{X^{1,b}}\non\\
   &\les& \frac{1}{L_{(1,j')}^+L_{(2,k')}^+(L_{(3,m')}L_{(4,n')})^{b-1/4-}}\|u_1^{j,j'}\|_{X^{0,b}}\|u_2^{k,k'}\|_{X^{2(1/4-b+),b}}\|u_3^{m,m'}\|_{X^{1,b}}\|u_4^{n,n'}\|_{X^{1,b}}\non\\
   &\les& \frac{1}{L_{(1,j')}^+L_{(2,k')}^+(L_{(3,m')}L_{(4,n')})^{b-1/4-}}\|u_1\|_{X^{0,b}}\|u_2\|_{X^{2(1/4-b+),b}}\|u_3\|_{X^{1,b}}\|u_4\|_{X^{1,b}}\non\\
   &\les& \frac{(N_{(1,j)}N_{(2,k)}N_{(3,m)}N_{(4,n)})^-}{L_{(1,j')}^+L_{(2,k')}^+(L_{(3,m')}L_{(4,n')})^{b-1/4-}}(N_{1,j}N_{3,m}N_{4,n})^+\|u_1\|_{X^{0,b}}\|u_2\|_{X^{2(1/4-b+),b}}\|u_3\|_{X^{1,b}}\|u_4\|_{X^{1,b}}\non\\
   &\les&\frac{(N_{(1,j)}N_{(2,k)}N_{(3,m)}N_{(4,n)})^-}{L_{(1,j')}^+L_{(2,k')}^+(L_{(3,m')}L_{(4,n')})^{b-1/4-}}(N_{2,k})^+\|u_1\|_{X^{0,b}}\|u_2\|_{X^{2(1/4-b+),b}}\|u_3\|_{X^{1,b}}\|u_4\|_{X^{1,b}}\non\\
   &\les&\frac{(N_{(1,j)}N_{(2,k)}N_{(3,m)}N_{(4,n)})^-}{L_{(1,j')}^+L_{(2,k')}^+(L_{(3,m')}L_{(4,n')})^{b-1/4-}}\|u_1\|_{X^{0,b}}\|u_2\|_{X^{2(1/4-b+),b}}\|u_3\|_{X^{1,b}}\|u_4\|_{X^{1,b}}\non\\
   &\les&\frac{(N_{(1,j)}N_{(2,k)}N_{(3,m)}N_{(4,n)})^-}{L_{(1,j')}^+L_{(2,k')}^+(L_{(3,m')}L_{(4,n')})^{b-1/4-}}\|u_1\|_{X^{0,b}}\|u_2\|_{X^{2(1/4-b+),b}}\|u_3\|_{X^{1,b}}\|u_4\|_{X^{1,b}}.\non
   \end{eqnarray}
 Since the summand is summable in $L$'s and $N$'s, we get the result for $\sigma=2(b-1/4)-$. Now we are left with the last case,
 
 \textbf{Case 2:} $|\tau_3| > 1/3 N_{(2,k)}^2$ or $|\tau_4| > 1/3 N_{(2,k)}^2$. Again, without loss of generality, we will only focus on $|\tau_3| >1/3 N_{(2,k)}^2$. In this case we have $$|\tau_3+\mu_{k_3}|\geq ||\tau_3|-|\mu_{k_3}||\geq 1/3 N_{(2,k)}^2- 4N_{(3,m)}^2 \geq 1/3N_{(2,k)}^2- 1/4N_{(2,k)}^2=1/12N_{(2,k)}^2,$$ which says $L_{(3,m)}\ges N_{(2,k)}^2$ and redoing the previous calculations, we obtain
   \begin{eqnarray}
   II(N,L)&\les& \frac{1}{(L_{(1,j')}L_{(2,k')}L_{(3,m')}L_{(4,n')})^{b-1/4-}}\|u_1^{j,j'}\|_{X^{0,b}}\|u_2^{k,k'}\|_{X^{0,b}}\|u_3^{m,m'}\|_{X^{1,b}}\|u_4^{n,n'}\|_{X^{1,b}}\non\\
   &\les& \frac{N_{(2,k)}^{2(1/4-b)+}}{(L_{(1,j')}L_{(2,k')}L_{(4,n')})^{b-1/4-}}L_{(3,m')}^+\|u_1^{j,j'}\|_{X^{0,b}}\|u_2^{k,k'}\|_{X^{0,b}}\|u_3^{m,m'}\|_{X^{1,b}}\|u_4^{n,n'}\|_{X^{1,b}}\non\\
   &\les& \frac{(N_{(1,j)}N_{(2,k)}N_{(3,m)}N_{(4,n)})^-}{(L_{(1,j')}L_{(2,k')}L_{(4,n')})^{b-1/4-}}L_{(3,m')}^+ N_{(2,k)}^{2(1/4-b+)}\|u_1^{j,j'}\|_{X^{0,b}}\|u_2^{k,k'}\|_{X^{0,b}}\|u_3^{m,m'}\|_{X^{1,b}}\|u_4^{n,n'}\|_{X^{1,b}}\non\\
   &\les& \frac{(N_{(1,j)}N_{(2,k)}N_{(3,m)}N_{(4,n})^-}{(L_{(1,j')}L_{(2,k')L_{(4,n')})^{b-1/4-}}} \|u_1^{j,j'}\|_{X^{0,b}}\|u_2^{k,k'}\|_{X^{2(1/4-b+),b}}\|u_3^{m,m'}\|_{X^{1,b}}\|u_4^{n,n'}\|_{X^{1,b}}\non\\
   &\les& \frac{(N_{(1,j)}N_{(2,k)}N_{(3,m)}N_{(4,n)})^-}{(L_{(1,j')}L_{(2,k')}L_{(4,n')})^{b-1/4-}}\|u_1\|_{X^{0,b}}\|u_2\|_{X^{2(1/4-b+),b}}\|u_3\|_{X^{1,b}}\|u_4\|_{X^{1,b}},\non
   \end{eqnarray}
   again the result follows for $\sigma=2(b-1/4)-$. Thus we have finished estimating $II$ and therefore Theorem \ref{growth}.
   \subsection*{Acknowledgement:}I would like to thank Nikolaos Tzirakis and Burak Erdo\~{g}an for their valuable comments. I would also like to thank to Gigliola Staffilani for bringing to my attention the work of Huxley.

 \end{document}